\documentclass[12pt, reqno]{amsart}
\usepackage{amsmath, amsthm, amscd, amsfonts, amssymb, graphicx, color}
\usepackage[bookmarksnumbered, colorlinks, plainpages]{hyperref}
     \makeatletter
     \def\section{\@startsection{section}{1}%
      \z@{.7\linespacing\@plus\linespacing}{.5\linespacing}%
     {\bfseries%\normalfont\scshape
     \centering
     }}
     \def\@secnumfont{\bfseries}
     \makeatother
\newtheorem{thm}{Theorem}[section]
\newtheorem{lem}[thm]{Lemma}
\newtheorem{prop}[thm]{Proposition}
\newtheorem{cor}[thm]{Corollary}
\theoremstyle{definition}

\theoremstyle{remark}
\newtheorem{rem}[thm]{Remark}
\numberwithin{equation}{section}

\newcommand{\vertiii}[1]{{\left\vert\kern-0.25ex\left\vert\kern-0.25ex\left\vert
#1 \right\vert\kern-0.25ex\right\vert\kern-0.25ex\right\vert}}

\def \e   {\text {\rm e}}

\def \In   {\text {\rm In}}

\def \qed   {\hfill \vrule height6pt width 6pt depth 0pt}

\begin{document}
\title[]{Hadamard powers of some positive matrices}

%\author[Rajendra Bhatia]{Rajendra Bhatia}
%\address{Indian Statistical Institute\\ New Delhi 110016, India}

%\email{rbh@isid.ac.in}

\author[Tanvi Jain]{Tanvi Jain}
\address{Indian Statistical Institute\\ New Delhi 110016, India}
\email{tanvi@isid.ac.in}



% \begin{left}
% (Work in progress)\\
% May 2014
% \end{left}

\begin{abstract}
Positivity properties of the Hadamard powers of the matrix $\begin{bmatrix}1+x_ix_j\end{bmatrix}$ for distinct positive real numbers $x_1,\ldots,x_n$
 and the matrix $\begin{bmatrix}|\cos((i-j)\pi/n)|\end{bmatrix}$ are studied. 
In particular, it is shown that $\begin{bmatrix}(1+x_ix_j)^r\end{bmatrix}$ is not positive semidefinite for any positive real number $r<n-2$
 that is not an integer, and $\begin{bmatrix}|\cos((i-j)\pi/n)|^r\end{bmatrix}$ is positive semidefinite for every odd integer $n\ge 3$ and $n-3\le r<n-2.$
\end{abstract}

\maketitle

\section{Introduction}
\vskip.2in
Positive definite matrices are fundamental objects of study in matrix analysis and have applications in diverse areas such as engineering, statistics, quantum information,
medical imaging and mechanics.
 A classical problem in matrix analysis involves the study of functions that act entrywise on matrices and preserve positivity. 
See, for instance, \cite{gkr, gkr1, h, ho, jw, r, s}. In particular, the study of entrywise power functions $t\mapsto t^r$ has been of special interest to various mathematicians; see \cite{be, fh, gkr, h}.

According to a theorem of Schur, the $m$th Hadamard power $A^{\circ m}=\begin{bmatrix}a_{ij}^m\end{bmatrix}$ of a positive semidefinite  matrix $A=\begin{bmatrix}a_{ij}\end{bmatrix}$ is again positive semidefinite for every positive integer $m.$ A positive semidefinite matrix is called {\it doubly nonnegative} if all its entries are nonnegative. If $A$ is a doubly nonnegative matrix and $r$ is a positive real number, then the $r$th {\it Hadamard power} of $A$ is the matrix $A^{\circ r}=\begin{bmatrix}a_{ij}^r\end{bmatrix}.$

FitzGerald and Horn \cite{fh} extensively studied the Hadamard powers of doubly nonnegative matrices.
 They showed that $n-2$ is the $`$critical exponent' for $n\times n$ doubly nonnegative matrices, i.e.,
$n-2$ is the least number for which $A^{\circ r}$ is doubly nonnegative for every $n\times n$ doubly nonnegative matrix $A$ and $r\ge n-2.$
 They also showed that for every positive real number $r<n-2$ that is not
an integer, we can find a doubly nonnegative matrix whose $r$th Hadamard power is not positive semidefinite.

If $A$ has arbitrary real (not necessarily nonnegative) entries, we consider a natural extension of real Hadamard powers.
For a real positive semidefinite matrix $A$ and a positive real number $r,$ we denote the matrix $\begin{bmatrix}|a_{ij}|^r\end{bmatrix}$ by $|A|_\circ^{\circ r}.$
In particular if $r=1,$ then we denote the matrix $\begin{bmatrix}|a_{ij}|\end{bmatrix}$ by $|A|_\circ.$
If $A$ is a $2\times 2$ real positive semidefinite matrix, then $|A|_\circ^{\circ r}$ is positive semidefinite for all positive real $r.$ 
But this is not true for higher dimensions. In fact, for every positive real $r$ that is not an even integer, 
we can find a real positive semidefinite matrix $A$ for which $|A|_\circ^{\circ r}$ is not positive semidefinite. (See \cite{be}).
When $r$ is an even positive integer, then $a_{ij}^r=|a_{ij}|^r.$ Hence by Schur's theorem, $|A|_\circ^{\circ r}$ is positive semidefinite in this case.
 Hiai \cite{h} proved an analogue of the theorem of FitzGerald and Horn for $n\times n$ real positive semidefinite matrices.
He showed that
for every $n\times n$ real positive semidefinite matrix $A,$ $n-2$ is the least number for which $|A|_\circ^{\circ r}$ is positive semidefinite for all $r\ge n-2.$
% in this case also, the critical exponent is $n-2.$

Recently there has been a renewed interest in the study of positivity properties of Hadamard powers of positive semidefinite matrices. 
Motivated by problems occurring in statistics, Guilllot, Khare and Rajaratnam have been studying various problems related to Hadamard powers. See \cite{gkr, gkr1}. 

If $r<n-2,$ there are two classes of examples in the literature.
 FitzGerald and Horn \cite{fh} considered the $n\times n$ doubly nonnegative matrix $A$ with $i,j$th entry $(1+\varepsilon ij)$ and showed that if $r$ is not an integer and $0<r<n-2,$ then $A^{\circ r}$ is not positive semidefinite for a sufficiently small positive number $\varepsilon.$ In this paper we show that this remains true if we replace $\varepsilon ij$ with $x_ix_j$ for any distinct positive real numbers $x_1,\ldots,x_n.$

\begin{thm}\label{thm1}
Let $x_1,\ldots,x_n$ be distinct positive real numbers. Let $X$ be the $n\times n$ matrix
\begin{equation}
X=\begin{bmatrix}1+x_ix_j\end{bmatrix}.\label{eq1}
\end{equation}
Then $X^{\circ r}$ is positive semidefinite if and only if $r$ is a nonnegative integer or $r>n-2.$
\end{thm}

Bhatia and Elsner \cite{be} studied another interesting class of $n\times n$ positive semidefinite Toeplitz matrices with real entries
\begin{equation}
C=\begin{bmatrix}\cos\bigl(\frac{(i-j)\pi}{n}\bigr)\end{bmatrix}.\label{eq2}
\end{equation}
They showed that for every even positive integer $n,$ the matrix $|C|_\circ^{\circ r}$ is not positive semidefinite if $n-4<r<n-2.$ The case when $n$ is odd remained open. In the next theorem we address this case.

\begin{thm}\label{thm2}
Let $n\ge 3$ be an odd integer and let $C$ be the $n\times n$ matrix given by \eqref{eq2}. Then $|C|_\circ^{\circ r}$ is positive semidefinite if and only if $r$ is a nonnegative even integer or $r>n-3.$
\end{thm}

We prove Theorem \ref{thm1} in Section 2. An essential ingredient of our proof comes from a recent analysis of the eigenvalue behaviour of a family of special matrices that was initiated in \cite{bj}. We continue with this analysis in Section 2, and prove some related results on the way. Theorem \ref{thm2} is proved in Section 3.
\vskip.2in
\section{Proof of Theorem \ref{thm1}}
\vskip.2in
Let $p_1,\ldots,p_n$ be distinct positive real numbers and let $r$ be a nonnegative real number. Let $P_r$ be the $n\times n$ real symmetric matrix
\begin{equation}
P_r=\begin{bmatrix}(p_i+p_j)^r\end{bmatrix}.\label{eq21}
\end{equation}
The numbers of positive and negative eigenvalues of $P_r$ were computed for all real $r$ in \cite{bj}. The next proposition follows from Theorem 1 of \cite{bj}.

\begin{prop}\label{prop3}
Let $P_r$ be the $n\times n$ matrix given by \eqref{eq21}. 
For each $r\ge 0$ the sign of the determinant of $P_r$ is given by $\varepsilon_{n,r}$ where
\begin{equation}
\varepsilon_{n,r}=\begin{cases}
(-1)^{\lfloor n/2\rfloor}  & \textrm{ if $r>n-2$}\\
(-1)^{n-k+1} & \textrm{ if $2k<r<2k+1\le n-2,$ $k\ge 0$ is an integer}\\
(-1)^{k+1} & \textrm{ if $2k+1<r<2k+2\le n-2,$ $k\ge 0$ is an integer}.
\end{cases}\label{eq22}
\end{equation}
If $r=0,1,\ldots,n-2,$ the determinant is zero. In this case, we take $\varepsilon_{n,r}=0.$
\end{prop}

A matrix $A$ is said to be {\it strictly sign regular (SSR) with signature} $\varepsilon=(\varepsilon_1,\ldots,\varepsilon_n)$ $(\varepsilon_k\in\{-1,1\})$ if 
\begin{align}
\textrm{every $k\times k$ subdeterminant of $A$ is nonzero with sign $\varepsilon_k,$}\label{eqr1}\tag{$*$}
\end{align}
 for all $1\le k\le n.$ Let $1\le m< n.$ Then the matrix is called {\it SSR}$_m$ with signature $\varepsilon=(\varepsilon_1,\ldots,\varepsilon_m)$ if the condition \eqref{eqr1}
holds for all $1\le k\le m.$ An SSR matrix is called {\it strictly totally positive (STP)} if all $\varepsilon_k$'s equal $1,$ i.e.,
all $k\times k$ subdeterminants are strictly positive. One can refer to \cite{fj, k} for a detailed study of these matrices.  

 Let $p_1<p_2<\cdots<p_n$ and $q_1<q_2<\cdots<q_n$ be positive real numbers. For $r$ in $\mathbb{R},$ let $S_r$ be the $n\times n$ matrix
\begin{equation}
S_r=\begin{bmatrix}(q_i+p_j)^r\end{bmatrix}.\label{eq23}
\end{equation}
In our next theorem we show that the $n\times n$ matrix $S_r$ is an SSR matrix for every real number $r\ne 0,1,\ldots,n-2.$

\begin{thm}\label{thm4}
Let $r$ be a nonnegative real number. The matrix $S_r,$ defined in \eqref{eq23}, is SSR if $r\ne 0,1,\ldots,n-2$ with signature $(\varepsilon_{k,r})_{k\le n},$ and is SSR$_r$ if $r=0,1,\ldots,n-2.$

In particular, the matrix $P_r$ given by \eqref{eq21} is SSR for $r\ne 0,1,\ldots,n-2$ and is SSR$_r$ whenever $r=0,\ldots,n-2.$
\end{thm}

To prove Theorem \ref{thm4}, we use the following extension of Theorem 4 of \cite{bj}.

\begin{prop}\label{prop5}
 Let $p_1,\ldots,p_n$ be distinct positive real numbers. Let $c_1, \ldots, c_n$ be real numbers, not all of which are zero. For each real number $r,$ define the function $f_r$ on $(0, \infty)$ as
\begin{equation}
f_r (x) = \sum_{j=1}^{n} \,\,c_j (x + p_j)^r. \label{eq4.7} 
\end{equation}
Then $f_r$ has at most $n-1$ zeros. (Here we follow the convention that the number of zeros of the identically zero function is zero.)
\end{prop}

\begin{proof}
We denote the number of zeros of the function $f_r$ by $Z(f_r),$ and the number of sign changes in the tuple $(c_1,\ldots,c_n)$ by $V(c_1,\ldots,c_n),$ see \cite{bj}. 
We prove the inequality

\begin{equation*}
Z(f_r)\le V(c_1,\ldots,c_n)\ \textrm{ for all }r\textrm{ in }\mathbb{R},
\end{equation*}
by induction on $V(c_1,\ldots,c_n).$ This can be proved by following arguments similar to those used in the proof of Theorem 4 of \cite{bj}.
We give a brief sketch, referring the reader to \cite{bj} for more details.

The case when $V(c_1,\ldots,c_n)=0$ is trivial. Assume that $Z(f_r)\le k-1$ whenever $V(c_1,\ldots,c_n)=k-1.$ Now suppose that $V(c_1,\ldots,c_n)=k,$ $k>0.$ We can assume that $c_i\ne 0$ for all $i=1,\ldots,n.$ Let $j$ be an index, $1<j\le n,$ such that $c_{j-1}c_j<0,$ and choose a real number $u$ such that $p_{j-1}<u<p_j.$ Consider the function
\begin{equation*}
\varphi (x) = \sum_{j=1}^n \,\,c_j (p_j - u) (x+p_j)^{r-1}. 
\end{equation*}
Note that $V(c_1(p_1-u), c_2(p_2-u),\ldots,c_n(p_n-u))=k-1.$ Hence by the induction hypothesis, $Z(\varphi)\le k-1.$
A computation reveals that
\begin{equation*}
\varphi eq338
x) = \frac{-(x+u)^{r+1}}{s} \,\, h^{\prime}(x),
\varphi (x) = \frac{-(x+u)^{r+1}}{s} \,\, h^{\prime}(x),
\end{equation*}
where $h(x) = \frac{f_r (x)}{(x+u)^r}.$ It is clear that $Z(\varphi)=Z(h^\prime),$ $Z(h)=Z(f_r),$ and by Rolle's theorem $Z(h)\le Z(h^\prime)+1.$ Combining these relations, we obtain $Z(f_r)\le Z(\varphi)+1\le k.$
\end{proof}

\noindent {\it Proof of Theorem \ref{thm4}}. Let $r\ne 0,1,\ldots,n-2.$ The matrix $S_r$ defined in \eqref{eq23} is singular if and only if there exist real numbers $c_1,\ldots,c_n,$ not all of which are zero, satisfying
\begin{equation*}
\sum\limits_{j=1}^{n}c_j(q_i+p_j)^r=0 \textrm{ for all }i=1,\ldots,n.
\end{equation*}
This is possible only when the function $f_r$ defined in \eqref{eq4.7} is either identically zero or has at least $n$ zeros $q_1,\ldots,q_n.$ The function $f_r$ is identically zero if and only if $f_r^{(k)}(x_0)=0$ for all nonnegative integers $k$ and for all $x_0>0.$ This is possible only when $r=0,1,\ldots,n-2.$ Hence $f_r$ is not identically zero. Also the number of zeros of $f_r$ is at most $n-1$ by Proposition \ref{prop5}. Hence $S_r$ is nonsingular.

Now by using a homotopy argument, continuity of the determinant function, and the intermediate value theorem, we obtain that the $n\times n$ matrix $S_r$ is nonsingular and its determinant has the same sign as that of the determinant of $P_r,$ i.e., $\varepsilon_{n,r},$ given by Proposition \ref{prop3}. Since each $k\times k$ submatrix of $S_r$ is again of this form, each $k\times k$ subdeterminant of $S_r$ is nonsingular and has sign $\varepsilon_{k,r}.$\qed
\vskip.2in

\begin{thm}\label{thmr1}
Let $p_1<p_2<\cdots<p_n,$ $q_1<q_2<\cdots<q_n$ be positive real numbers and let $r$ be a nonnegative real number. Consider the $n\times n$ matrix
\begin{equation*}
H_r=\begin{bmatrix}(1+q_ip_j)^r\end{bmatrix}.
\end{equation*}
Then for each $r\ne 0,1,\ldots,n-2,$ $H_r$ is SSR with signature $(\varepsilon^\prime_{k,r})_{k=1}^{n},$ where $\varepsilon^\prime_{k,r}=(-1)^{\lfloor k/2\rfloor}\varepsilon_{k,r}.$
\end{thm}

\begin{proof}
We can write $H_r$ as
\begin{eqnarray*}
H_r & = & \begin{bmatrix}(1+q_ip_j)^r\end{bmatrix}\\
 & = & \begin{bmatrix}q_i^r\bigl(\frac{1}{q_i}+p_j\bigr)^r\end{bmatrix}\\
& = & D^r\begin{bmatrix}\bigl(\frac{1}{q_i}+p_j\bigr)^r\end{bmatrix}.
\end{eqnarray*}
Here $D$ is the positive diagonal matrix with entries $q_1,\ldots,q_n$ on its diagonal. We have $1/q_1>1/q_2>\cdots>1/q_n.$ Put $s_j=1/q_{n-j+1},$ $j=1,\ldots,n.$ Then $s_1<s_2<\cdots<s_n.$ Also
\begin{eqnarray*}
H_r & = & D^r\begin{bmatrix}(s_{n-i+1}+p_j)^r\end{bmatrix}\\
 & = & D^rV\begin{bmatrix}(s_i+p_j)^r\end{bmatrix},
\end{eqnarray*}
where $V$ is the antidiagonal matrix with its antidiagonal consisting of all ones. The determinant of $V$ is $(-1)^{\lfloor n/2\rfloor}.$ Hence, by Theorem \ref{thm4}, if $r\ne 0,1,\ldots,n-2,$  then $H_r$ is nonsingular and its determinant has sign $\varepsilon^\prime_{n,r}=(-1)^{\lfloor n/2\rfloor}\varepsilon_{n,r}.$ Now the theorem follows from the fact that every $k\times k$ submatrix of $H_r$ is again of the form $H_r.$
\end{proof}
\vskip.2in
\noindent {\it Proof of Theorem \ref{thm1}}. Without loss of generality we can assume that $x_1<\cdots <x_n.$ 
Let $r$ be a positive real number. The matrix $X^{\circ r}$ is a special case of $H_r$ when $p_i=q_i=x_i.$ 
Hence $X^{\circ r}$ is SSR with signature $(\varepsilon^\prime_{k,r})_{k=1}^{n}$ when $r\ne 0,1,\ldots,n-2.$ 
If $r>n-2,$ then $X^{\circ r}$ is positive definite.
Now let $r<n-2.$

Suppose that $2k<r<2k+1\le n-2$ for some nonnegative integer $k.$ Then the sign of every $(2k+3)\times (2k+3)$ subdeterminant of $X^{\circ r}$ is 
\begin{align}
\varepsilon^\prime_{2k+3,r}=(-1)^{k+1}\varepsilon_{2k+3,r}=(-1)^{k+1}(-1)^{k+4}=-1.\label{eqrr1}
\end{align}
Thus $X^{\circ r}$ is not positive semidefinite in this case. A similar calculation shows that $X^{\circ r}$ is not positive semidefinite for $2k+1<r<2k+2\le n-2.$ \qed
\vskip.2in

\begin{rem}
If the numbers $x_1,\ldots,x_n$ are arranged in either increasing or decreasing order, then the matrix $X^{\circ r}$ is STP for $r>n-2$ and is not positive definite when $0\le r\le n-2.$
\end{rem}

The {\it inertia} of a Hermitian matrix $A$ is the ordered triple $\In\, A=(\rho,\zeta,\eta),$ where $\rho, \zeta,\eta,$
respectively, denote the number of positive, zero and negative eigenvalues of $A.$

\begin{thm}
Let $x_1,\ldots,x_n$ be positive distinct numbers, and let $X$ be the $n\times n$ matrix defined in \eqref{eq1}.
\begin{itemize}
\item[$(i)$] $X^{\circ r}$ is positive definite for every $r>n-2,$ and $\In\, X^{\circ r}=(n,0,0).$
\item[$(ii)$] If $r=0,1,\ldots,n-2,$ then $X^{\circ r}$ is positive semidefinite with rank $r+1,$ and $\In\, X^{\circ r}=(r+1,n-(r+1),0).$
\item[$(iii)$] If $0\le m<r<m+1\le n-2,$ then $\In\, X^{\circ r}=(n-\lceil(n-m-2)/2\rceil, 0, \lceil\frac{n-m-2}{2}\rceil).$
\item[$(iv)$] All the eigenvalues of $X^{\circ r}$ are simple.
\end{itemize}
\end{thm}

\begin{proof}
$(i)$ follows from Theorems \ref{thm1} and \ref{thmr1}.
Let $r=0,1,\ldots,n-2.$ Then
\begin{eqnarray*}
X^{\circ r} & = & \begin{bmatrix}(1+x_ix_j)^r\end{bmatrix}\\
 & = & \begin{bmatrix}1+\binom{r}{1}x_ix_j+\cdots +\binom{r}{r-1}x_i^{r-1}x_j^{r-1}+x_i^rx_j^r\end{bmatrix}\\
 & = & W^*D_0W,
\end{eqnarray*}
where $W$ is the $(r+1)\times n$ Vandermonde matrix
\begin{equation*}
\begin{bmatrix}
1 & 1 & \cdots & 1\\
x_1 & x_2 & \cdots & x_n\\
\vdots & \vdots & \vdots\vdots\vdots & \vdots\\
x_1^r & x_2^r & \cdots & x_n^r\end{bmatrix}
\end{equation*}
and $D_0$ is the $(r+1)\times (r+1)$ diagonal matrix with its $j$th diagonal entry $\binom{r}{j-1}.$
Since $x_i\ne x_j$ for $i\ne j,$ $W$ has rank $r+1.$ The matrix $D_0$ is invertible. Hence $X^{\circ r}$ has rank $r+1.$
Therefore $\In\, X^{\circ r}=(r+1,n-r-1,0).$

Let $m<r<m+1\le n-2.$
Without loss of generality, we can assume that $x_1<\cdots<x_n.$
By Theorem \ref{thmr1} $X^{\circ r}$ is SSR with signature $(\varepsilon^\prime_{k,r})_{k=1}^{n}.$
If $k=1,\ldots,m+2,$ then $\varepsilon^\prime_{k,r}=1.$ 
By calculations similar to \eqref{eqrr1} we obtain
\begin{equation*} 
   \varepsilon^\prime_{m+2+l,r}=(-1)^{\lceil l/2\rceil},\ \ \ l=1,2,\ldots,n-m-2.
	\end{equation*}
	To prove $(iii)$ and $(iv)$, we use arguments similar to those used in \cite{bfj}.
	Since each entry of the $k$th exterior power $\Lambda^kX^{\circ r}$ is a $k\times k$ subdeterminant of $X^{\circ r}$ and $X^{\circ r}$ is SSR,
	all the entries of $\Lambda^kX^{\circ r}$ have the same sign $\varepsilon^\prime_{k,r}.$ 
	Let $\lambda_1,\ldots,\lambda_n$ be the eigenvalues of $X^{\circ r}$ arranged so that 
	$|\lambda_1|\ge |\lambda_2|\ge\cdots\ge |\lambda_n|.$
	Then $\lambda_1\lambda_2\cdots\lambda_k$ is the eigenvalue of $\Lambda^kX^{\circ r}$ that has the maximum modulus. 
	Perron's theorem tells us that if $A$ is an entrywise positive square matrix, then the eigenvalue of $A$ that has 
	the maximum modulus is positive and simple.
	Since the entries of $\Lambda^kX^{\circ r}$ have sign $\varepsilon^\prime_{k,r},$
	the eigenvalue $\lambda_1\lambda_2\cdots\lambda_k$ is simple and has sign $\varepsilon^\prime_{k,r}.$ 
	Thus all the eigenvalues $\lambda_1,\lambda_2,\ldots,\lambda_n$ are simple.
	For $k=1,\ldots,m+2$ $X^{\circ r}$ is entrywise positive.
	Thus $\lambda_1,\lambda_2,\ldots,\lambda_{m+2}$ are positive.
Since $\lambda_1\cdots\lambda_{m+3}<0,$ $\lambda_{m+3}<0.$
Now $\lambda_1\cdots\lambda_{m+3}\lambda_{m+4}<0$ implies $\lambda_{m+4}>0,$ and
$\lambda_1\cdots\lambda_{m+4}\lambda_{m+5}>0$ implies $\lambda_{m+5}<0.$
Continuing in this manner for each $j=1,\ldots,n-m-2$ we have
\begin{equation*}
\lambda_{m+2+j}<0\textrm{ if $j$ is odd, and }\lambda_{m+j}>0\textrm{ if $j$ is even.}
\end{equation*}
Thus the number of negative eigenvalues of $X^{\circ r}$ is $\lceil(n-m-2)/2\rceil,$ and the number of positive
eigenvalues is $n-\lceil (n-m-2)/2\rceil.$
\end{proof}

\begin{cor}\label{corr2}
Let $p_1,\ldots,p_n$ and $q_1,\ldots,q_n$ be positive real numbers such that $p_i/q_i\ne p_j/q_j$ if $i\ne j.$ For $r>0,$ the $n\times n$ matrix
 $\begin{bmatrix}(p_ip_j+q_iq_j)^r\end{bmatrix}$ is positive definite if and only if $r>n-2.$
\end{cor}

\begin{proof}
Let $s_j=p_j/q_j,$ $j=1,\ldots,n.$ The proof follows from the congruence $\begin{bmatrix}(p_ip_j+q_iq_j)^r\end{bmatrix}=D^r\begin{bmatrix}1+(s_is_j)^r\end{bmatrix}D^r,$ where $D$ is the diagonal matrix with diagonal entries $q_1,\ldots,q_n.$
\end{proof}

\begin{cor}
Let $0<x_1<x_2\cdots<x_n<\pi/2,$ and let $r$ be a positive real number. Then the $n\times n$ matrix $\begin{bmatrix}\cos ^r(x_i-x_j)\end{bmatrix}$ is positive
definite if and only if $r>n-2.$
\end{cor}

\begin{proof}
Take $p_i=\cos\, x_i$ and $q_i=\sin\, x_i,$ $1\le i\le n.$ Since $0<x_1<\cdots <x_n<\pi/2,$ $p_i$ and $q_i$ satisfy the conditions of Corollary \ref{corr2}.
The proof follows from the identity $\cos(x_i-x_j)=\cos\, x_i\, \cos\, x_j+\sin\, x_i\, \sin\, x_j,$ and Corollary \ref{corr2}.
\end{proof}

\section{Proof of Theorem \ref{thm2}}
\vskip.2in
\begin{lem}\label{lem31}
Let $n$ and $k$ be integers such that $n$ is odd and $0\le 2k<n.$
Then exactly $n-2k-1$ eigenvalues of the $n\times n$ matrix $|C|_\circ^{\circ 2k}$ are zero.
\end{lem}

\begin{proof}
We can write the $n\times n$ matrix $|C|_\circ^{\circ 2k}$ as
\begin{eqnarray*}
|C|_\circ^{\circ 2k} & = & \begin{bmatrix}\bigl(\cos\frac{(i-j)\pi}{n}\bigr)^{2k}\end{bmatrix}\\
 & = & \begin{bmatrix}\bigl(\cos\frac{i\pi}{n}\cos\frac{j\pi}{n}+\sin\frac{i\pi}{n}\sin\frac{j\pi}{n}\bigr)^{2k}\end{bmatrix}.
\end{eqnarray*}
Let $D$ be the $n\times n$ diagonal matrix with its $j$th diagonal entry $\cos^{2k}(j\pi/n),$ and let $z_j=\tan(j\pi/n).$ Since $n$ is odd, $\cos(j\pi/n)\ne 0$ for every $j=1,\ldots,n.$ Thus $D$ is invertible and each $z_j$ is finite.
We have
\begin{eqnarray*}
|C|_\circ^{\circ 2k} & = & \begin{bmatrix}\cos^{2k}\frac{i\pi}{n}\bigl(1+z_iz_j)^{2k}\cos^{2k}\frac{j\pi}{n}\end{bmatrix}\\
 & = & D\begin{bmatrix}(1+z_iz_j)^{2k}\end{bmatrix}D\\
 & = & D\begin{bmatrix}1+\binom{2k}{1}z_iz_j+\cdots+\binom{2k}{2k-1}z_i^{2k-1}z_j^{2k-1}+z_i^{2k}z_j^{2k}\end{bmatrix}D\\
 & = & DW^*D_0WD,
\end{eqnarray*}
where $W$ is the $(2k+1)\times n$ Vandermonde matrix
\begin{equation*}
\begin{bmatrix}1 & 1 & \cdots & 1\\
z_1 & z_2 & \cdots & z_n\\
\vdots & \vdots & \vdots\vdots\vdots & \vdots\\
z_1^{2k} & z_2^{2k} & \cdots & z_n^{2k}\end{bmatrix},
\end{equation*}
and $D_0$ is the $(2k+1)\times (2k+1)$ diagonal matrix with its $j$th diagonal entry $\binom{2k}{j-1},$ $j=1,\ldots,2k+1.$ Since $z_i\ne z_j$ for all $i\ne j,$ $W$ is a full rank matrix. By the hypothesis $2k< n,$ the rank of $W$ is $2k+1.$
The matrices $D$ and $D_0$ are invertible. Hence the rank of $|C|_\circ^{\circ 2k}$ is $2k+1.$
Therefore the number of zero eigenvalues of $|C|_\circ^{\circ 2k}$ is $n-2k-1.$
\end{proof}

We shall use the following generalized version of the Descartes rule of signs, see pp. 46, Problem 77 in \cite{ps}.

\begin{prop}\label{prop32}
Let $a_1,\ldots,a_n,\mu_1,\ldots,\mu_n$ be real numbers such that $\mu_1>\mu_2>\cdots >\mu_n.$ Let $f$ be the function defined on
 $\mathbb{R}$ as
\begin{equation*}
f(x)=a_1\e^{\mu_1 x}+a_2\e^{\mu_2 x}+\cdots +a_n\e^{\mu_n x}.
\end{equation*}
Then the number of real zeros of $f$ is at most the number of sign changes in the tuple $(a_1,\ldots,a_n).$
\end{prop}

We shall use the fact that the eigenvalues of an $n\times n$ circulant matrix
\begin{equation*}
A=\begin{bmatrix}c_0 & c_1 & \cdots & c_{n-1}\\
c_{n-1} & c_0 & \cdots & c_{n-2}\\
\vdots & \vdots & \vdots\vdots\vdots & \vdots\\
c_1 & c_2 & \cdots & c_0\end{bmatrix}.
\end{equation*}
are given by
\begin{equation*}
c_0+w^jc_1+w^{2j}c_2+\cdots +w^{(n-1)j}c_{n-1},\ \ \ 0\le j\le n-1,
\end{equation*}
where $w=\e^{\iota 2\pi/n}.$
\vskip.2in
\noindent {\it Proof of Theorem \ref{thm2}}. 
Let $n=2m+1$ be an odd integer and let $r\ge 0.$
Using the fact that $|\cos(\pi-x)|=|\cos\, x|,$ we see that $|C|_\circ^{\circ r}$ is a circulant matrix.
Hence the eigenvalues of $|C|_\circ^{\circ r}$ are given by 
\begin{eqnarray*}
\lambda_j(r) & = & 1+\sum\limits_{k=1}^{n-1}w^{kj}|\cos(k\pi/n)|^r\\
 & = & 1+\sum\limits_{k=1}^{m}(w^{kj}+w^{-kj})\, \cos^r(k\pi/n)\\
 & = & 1+2\sum\limits_{k=1}^{m}\cos(2kj\pi/n)\, \cos^r(k\pi/n)\\
 & = & 1\cdot\e^{\mu_0r}+2\sum\limits_{k=1}^{m}\cos(2kj\pi/n)\e^{\mu_kr},
\end{eqnarray*}
where $\mu_k=\log\, \cos(k\pi/n),$ $k=0,\ldots,m.$ Note that $\mu_0>\mu_1>\cdots >\mu_m.$
		Let $Z_j$ denote the number of zeros of the function $\lambda_j(r)$ as $r$ varies over $[0,\infty).$
		By Proposition \ref{prop32} $Z_j$ is not bigger than the number of sign changes in the $m+1$-tuple 
\begin{equation*}
\zeta_j=(1,2\cos(2j\pi/n),\ldots,2\cos(2mj\pi/n)).
\end{equation*}

Since $\zeta_0=(1,2,\ldots,2),$ $Z_0=0.$ 
For $j=1,\ldots,2m$ let $\theta_{j,k}=2kj\pi/(2m+1),$ $k=0,1\ldots,m.$ Then 
\begin{equation*}
\zeta_j=(\cos\, \theta_{j,0},\cos\, \theta_{j,1},\ldots,\cos\, \theta_{j,m}).
\end{equation*}

Let $1\le j\le m.$ Note that $\cos\, \theta$ changes sign whenever $\theta$ crosses an odd multiple of $\pi/2.$ So to count the number of sign changes of $\zeta_j$ we need to track the number of times the argument $\theta_{j,k}$ moves past one such value.
Since $(2j-1)\pi/2<\theta_{j,m}<j\pi$ and $0=\theta_{j,0}<\theta_{j,1}<\cdots<\theta_{j,m},$
$\theta_{j,k}$'s cross the values $\pi/2,3\pi/2,\ldots,(2j-1)\pi/2$ exactly once, i.e.,
for each $p=1,2,\ldots,j$ there exists a unique $k$ that satisfies $\theta_{j,(k-1)}<(2p-1)\pi/2<\theta_{j,k}.$
Hence the number of sign changes in $\zeta_j$ is $j,$ and by Proposition \ref{prop32} $Z_j\le j.$

Now let $m+1\le j\le 2m.$ Let $j^\prime=2m+1-j.$ Then $1\le j^\prime\le m$ and $\theta_{j^\prime ,k}=2k\pi-\theta_{j,k}.$
This gives $\cos\, \theta_{j^\prime, k}=\cos\, \theta_{j,k},$ and consequently $\zeta_{j^\prime}=\zeta_j.$
This implies 
\begin{equation}
\lambda_j(r)=\lambda_{j^\prime}(r) \textrm{ for all } r\ge 0.\label{eq37}
\end{equation}
 Hence $Z_j=Z_{j^\prime}\le 2m+1-j.$

The total number of zero eigenvalues of $|C|_\circ^{\circ r}$ as $r$ varies over $[0,\infty)$ is at most
\begin{equation}
\sum\limits_{j=0}^{2m}Z_j\le \sum\limits_{j=0}^{m}j+\sum\limits_{j=m+1}^{2m}(2m+1-j)=m(m+1).\label{eq33}
\end{equation}
By Lemma \ref{lem31} we know that the number of zero eigenvalues of $(2m+1)\times (2m+1)$ matrix $|C|_\circ^{\circ 2k}$ is $2(m-k),$ $k=0,1,\ldots,m-1.$
Hence the total number of zero eigenvalues of $|C|_\circ^{\circ 2k}$ as $k$ varies over $0,1,\ldots,m-1$ is
\begin{equation}
\sum\limits_{k=0}^{m-1}2(m-k)=m(m+1).\label{eq34}
\end{equation}
Hence for any $j=0,1,\ldots,2m$ $\lambda_j(r)=0$ if and only if $r$ is an even integer $2k,$ $0\le k\le m-1.$
Therefore, the $(2m+1)\times (2m+1)$ matrix $|C|_\circ^{\circ r}$ is nonsingular for all $r\ne 0,2,\ldots,2m-2.$

If $r$ is an even integer, then $|C|_\circ^{\circ r}$ is positive semidefinite.
Let $r>2m-2.$ Theorem 5.1 of \cite{h} tells us that $|C|_\circ^{\circ r}$ is positive semidefinite for all $r\ge 2m-1.$
If $|C|_\circ^{\circ r}$ is not positive semidefinite for some $r=r_0$ in the interval $(2m-2,2m-1),$ then at least one eigenvalue
$\lambda_j(r_0)$ should be negative. 
Since $\lambda_j$ is a continuous function of $r,$ there exists an $s>r_0$ such that $\lambda_j(s)=0.$
But this cannot be so because $|C|_\circ^{\circ r}$ is nonsingular for all $r>2m-2.$ Hence $|C|_\circ^{\circ r}$ is positive definite for all $r>2m-2.$

Let $r$ be not an even integer and let $0<r<2m-2.$
From \eqref{eq33} and \eqref{eq34}, we see that 
\begin{equation}
Z_j=j=Z_{2m+1-j} \textrm{ for }j=0,1,\ldots,m.\label{eq36}
\end{equation}
Let $2\le j\le m.$ By Lemma \ref{lem31} and \eqref{eq36}, we see that 
$\lambda_j(r)=0$ if and only if $r=0,2,\ldots,2(j-1).$
Since $|C|_\circ^{\circ r}$ is positive definite for $r>2m-2,$ each $\lambda_j(r)>0$ for $r>2m-2.$ 
The last sign change of $\lambda_j(r)$ occurs at $r=2(j-1).$
Thus 
\begin{equation}
\lambda_j(r)>0 \textrm{ for }r>2(j-1), \textrm{ and }\lambda_j(r)<0 \textrm{ for }2(j-2)<r<2(j-1).\label{eqrs1}
\end{equation}
This implies that $|C|_\circ^{\circ r}$ is not positive semidefinite
whenever $r\in (2(j-2),2(j-1)),$ $j=2,3,\ldots,m.$ 
 \qed
\vskip.2in
Finally we give the inertia of $|C|_\circ^{\circ r}$ in the following theorem.

\begin{thm}
Let $n=2m+1$ be an odd positive integer and let $C$ be the $n\times n$ matrix defined in \eqref{eq2}. 
\begin{itemize}
\item[$(i)$] $|C|_\circ^{\circ r}$ is positive definite for $r>n-3,$ and $\In\, |C|_\circ^{\circ r}=(n,0,0).$
\item[$(ii)$] If $r$ is a nonnegative even integer less than $n,$ then $\In\, |C|_\circ^{\circ r}=(r+1,n-(r+1),0).$
\item[$(ii)$] If $0\le 2k<r<2k+1<n-3,$ then 
\begin{equation*}
\In\, |C|_\circ^{\circ r}=\begin{cases}
(m+k+1,0,m-k) & \textrm{ if $m-k$ is even}\\
(m+k,0,m+k+1) & \textrm{ if $m-k$ is odd.}
\end{cases}
\end{equation*}
\item[$(iv)$] Exactly one eigenvalue of $|C|_\circ^{\circ r}$ is simple.
\end{itemize}
\end{thm}

\begin{proof}
$(i)$ follows from Theorem \ref{thm2}, and $(ii)$ follows from Theorem \ref{thm2} and Lemma \ref{lem31}.

Let $0\le 2(k-1)<r<2k\le n-3.$
By \eqref{eqrs1}, $\lambda_j(r)>0$ for $j=0,1,\ldots,k,$ and $\lambda_{k+1}(r)<0.$
A sign change of $\lambda_j(r)$ occurs at each of its zero.
Since the only zeros of $\lambda_j(r)$ are $0,2,\ldots,2(j-1),$
 $\lambda_{k+2}(r)>0.$
We argue in a similar manner for $\lambda_{k+3}(r),\ldots,\lambda_m(r).$ 
thus for each $j=1,\ldots,m-k$ $\lambda_{k+j}(r)<0$ if $j$ is odd, and $\lambda_{k+j}>0$ if $j$ is even.
By \eqref{eq37} we have $\lambda_{2m+1-j}(r)=\lambda_j(r)$ for each $j=1,\ldots,m.$
Hence the number of negative eigenvalues of $|C|_\circ^{\circ r}$ is 
\begin{equation*}
2\lceil (m-k)/2\rceil=\begin{cases}
m-k & \textrm{ if $m-k$ is even}\\
m-k+1 & \textrm{ if $m-k$ is odd.}
\end{cases}
\end{equation*}
Again by \eqref{eq37}, we know that all the eigenvalues except $\lambda_0(r)$ are repeated at least twice.
Since $|C|_\circ^{\circ r}$ is entrywise positive, at least one eigenvalue must be simple.
Hence $\lambda_0(r)$ is the only simple eigenvalue.
\end{proof}

\vskip.2in
\noindent {\it Acknowledgment.} The work of the author is supported by a SERB Women Excellence Award.
The author thanks Rajendra Bhatia and Roger Horn for their valuable suggestions.

\vskip.2in
\noindent {\it Acknowledgment.} The work of the author is supported by a SERB Women Excellence Award.
The author thanks Rajendra Bhatia and Roger Horn for their valuable suggestions.

\end{document}